\documentclass[letterpaper, 11pt, conference,onecolumn]{ieeeconf}
\usepackage{amssymb}
\usepackage{amsmath}
\usepackage{epsfig}
\usepackage{graphicx}
\usepackage{amsfonts}
\usepackage{booktabs}
\usepackage{bigstrut}
\usepackage{multirow}
\usepackage{array}
\usepackage{tabularx}
\usepackage{float}
\IEEEoverridecommandlockouts
\newtheorem{definition}{Definition}
\newtheorem{theorem}{Theorem}

\newtheorem{lemma}{Lemma}

\renewcommand{\text}[1]{\mbox{\rm #1}}

\newcommand{\remove}[1]{}
\renewcommand{\arraycolsep}{0pt}

\begin{document}

\title{Solving A Class of Discrete Event Simulation-based Optimization Problems \\Using ``Optimality in Probability''}

\author{\authorblockN{Jianfeng Mao}
\authorblockA{School of Mechanical and Aerospace Engineering\\ Nanyang Technological University, Singapore 639798\\
Email: jfmao@ntu.edu.sg}\thanks{{\footnotesize The authors’ work is supported in part by ATMRI under Grant M4061216.057, by NTU under startup grant M58050030, by AcRF under Tier 1 grant RG 33/10 M52050117, by NSF under grants CNS-1239021, ECCS-1509084, and IIP-1430145, by AFOSR under grant FA9550-15-1-0471, and by ONR under grant N00014-09-1-1051. }}
\and
\authorblockN{Christos G. Cassandras}
\authorblockA{Division of Systems Engineering\\ Boston University, Brookline, MA 02446, USA\\
Email: cgc@bu.edu}}
\date{}
\maketitle

\begin{abstract}
We approach a class of discrete event simulation-based optimization problems using \emph{optimality in probability}, an approach which yields what is termed a ``champion solution''. Compared to the traditional \emph{optimality in expectation}, this approach favors the solution whose actual performance is more likely better than that of any other solution; this is an effective alternative to the traditional optimality sense, especially when facing a dynamic and nonstationary environment. Moreover, using \emph{optimality in probability} is computationally promising for a class of discrete event simulation-based optimization problems, since it can reduce computational complexity by orders of magnitude compared to general simulation-based optimization methods using \emph{optimality in expectation}. Accordingly, we have developed an “Omega Median Algorithm” in order to effectively obtain the champion solution and to fully utilize the efficiency of well-developed off-line algorithms to further facilitate timely decision making. An inventory control problem with nonstationary demand is included to illustrate and interpret the use of the Omega Median Algorithm, whose performance is tested using simulations.


\end{abstract}

\section{Introduction}

A general stochastic optimization problem using optimality in expectation can be formulated as
\begin{equation}\label{SOP}
\min\limits_{u\in \Phi} E [J(u,\omega)]
\end{equation}
where $u$ is the decision variable, $\Phi$ is the feasible space of $u$, and
$\omega$ is used to index sample paths resulting from different realizations
of a collection of random variables that affect the performance $J(u,\omega)$.
In the context of discrete event systems, we commonly face a dynamic
stochastic process, in which $u$ is an event-triggered online control action
and $J(u,\omega)$ is the actual performance of $u$ over a certain sample path
$\omega$. For example, in the on-line inventory control problem later
considered in Section \ref{OICP}, $u$ is the order quantity decided at the
beginning of each period, $\omega$ is a sample path constructed by a sequence
of demands, and $J(u,\omega)$ is the corresponding operating cost, including
setup cost, holding cost and shortage cost.

Since it is typically impossible to derive the closed form of
$E\big\{J(u,\omega)\big\}$ in (\ref{SOP}), simulation-based optimization
methods need to be employed to obtain a near-optimal solution. In what
follows, we define an \textquotedblleft evaluation\textquotedblright\ as an
operation of calculating the value of $J(u,\omega)$ for a specific $u$ over a
specific sample path $\omega$. In general, simulation-based optimization
methods include two major operations:
\begin{enumerate}
\item \textbf{Solution Assessment}: Implement $M$ evaluations for a specific
$u$ over $M$ sample paths and estimate the expected performance of solution
$u$, $E[J(u,\omega)]$, by sample average approximation, \emph{i.e.},
$\sum_{i=1}^{M}J(u,\omega_{i})\Big/M$;
\item \textbf{Search Strategy}: Use the sample average approximation in 1) to
rank solutions and search for better solutions in promising areas according to
gradient information or certain partition structures.
\end{enumerate}

Let $I$ denote the total number of solutions explored in a simulation-based
method and $C$ denote the complexity of an evaluation. Then, the total
complexity can be measured by the computational effort of implementing $M\cdot
I$ evaluations, that is, $O(M\cdot I\cdot C)$ ($M$ is not necessarily a
constant throughout the entire search process). To get a near optimal (or good
enough) solution, we need to implement more evaluations to refine solution
assessment, \emph{i.e.,} larger $M$, and explore a greater number of solutions,
\emph{i.e.,} larger $I$. Since both $M$ and $I$ can be very large in solving a
general simulation-based optimization problem using optimality in expectation,
this approach is computationally intensive or even intractable for many
applications in practice.

Some simulation-based optimization methods have been developed over the past
few decades. Computational effort can be reduced by either using a smaller
number $M$ of evaluations in assessment, such as Ordinal Optimization
\cite{HoZhaoJia08} and Optimal Computing Budget Allocation \cite{ChenLee11},
or by reducing $I$ in search, such as Nested Partitions \cite{ShiOlafsson00}
and COMPASS \cite{HongNelson06}, or by both ways, such as Perturbation
Analysis \cite{HoCao91} and Retrospective Optimization \cite{Chen94}%
\cite{Jin98}. Moreover, to further improve computational efficiency, these
methods may be applied to certain approximations of the original systems with
little loss of accuracy in the optimization solutions, such as the use of
Stochastic Flow Models \cite{CassandrasWardi02}\cite{YaoCassandras12} and
Hindsight Optimization \cite{ChongGivanChang00} \cite{WuChongGivan02}. Since
these methods still need to employ sample average approximations to assess
every explored solution (or estimate its performance gradient), their
complexity can still be approximated as $O(M\cdot I\cdot C)$ with either
smaller $M$ or smaller $I$ or both. In practice, timely decision making is
usually preferable or required in a dynamic environment. The heavy
computational burden of those methods using optimality in expectation limits
their applications in such situations.

Moreover, we argue that optimality in expectation is not truly
\textquotedblleft optimal\textquotedblright\ in certain cases since the
expected performance is not exactly the actual performance, but only a
promising guess. This kind of optimality is generally suitable for a
stationary environment, in which probability distributions remain unchanged
over time and the objective value is the average performance over the long
term. However, in practice we often face a nonstationary environment, such as
the example included in the paper, in which nonstationary demand is a common
occurrence in industries with short product life cycles, seasonal patterns,
varying customer behavior, or other factors. When we continually or
periodically make decisions, the probability distributions used are only valid
for a short term and need to be occasionally updated. Clearly, optimality in
expectation does not necessarily lead to the \textquotedblleft
best\textquotedblright\ solution in this case.

In this paper, we propose an alternative sense of optimality,
\textquotedblleft optimality in probability\textquotedblright, which favors a
solution that has a \emph{higher chance} to get a better actual performance.
The best solution using optimality in probability, termed \textquotedblleft
Champion Solution\textquotedblright, is defined as the one whose actual
performance is more likely better than that of any other solution.
\emph{Optimality in probability} is an effective alternative to optimality in
expectation, especially when facing a dynamic and nonstationary environment.
Moreover, using \emph{optimality in probability} is computationally promising
for a class of simulation-based optimization problems, since it can reduce
computational complexity by orders of magnitude compared to general
simulation-based optimization methods using \emph{optimality in expectation}.
Accordingly, we develop an \textquotedblleft Omega Median
Algorithm\textquotedblright\ to obtain the champion solution without
iteratively searching for better solutions based on sample average
approximations, a process which is computationally intensive and commonly
required when seeking optimality in expectation. Furthermore, although it is
quite challenging to solve many stochastic optimization problems, their
corresponding deterministic versions, which can be regarded as optimization
problems defined over a single sample path, have been efficiently solved by
certain off-line algorithms. The Omega Median Algorithm is able to fully
utilize the efficiency of these well-developed off-line algorithms to further
facilitate timely decision making, which is clearly preferable in a dynamic
environment with limited computing resources.

In the rest of the paper, we first introduce the champion solution and then
develop an efficient simulation-based optimization method, termed Omega Median
Approximation in Section II. We then consider a nonstationary inventory
control in Section III. Numerical results are given in Section IV to
demonstrate the performance of the champion solution. We close with
conclusions in Section V.

\section{Champion Solution}

The \textquotedblleft Champion Solution\textquotedblright\ is the best
solution using optimality in probability and defined for general stochastic
minimization problems as follows, where $\Pr[\cdot]$ is the usual notation for
\textquotedblleft probability\textquotedblright:
\begin{definition}
The champion solution is a solution $u^{c}$ such that
\begin{equation} \label{CSDef}
\Pr\left[ J(u^c,\omega)\le J(u,\omega)\right] \ge 0.5, \;\; \forall\; u \in \Phi,
\end{equation}
where $J(u,\omega)$ is the actual performance of $u$ over a certain sample
path $\omega$.
\end{definition}

\noindent\textbf{Remark}: A natural question which immediately arises is
\textquotedblleft why do we select $0.5$?\textquotedblright\ rather than some
$q>0.5$ and define the champion solution as $u^{\prime}$ below such that
\begin{equation} \label{CSDefw}
\Pr\left[ J(u',\omega)\le J(u,\omega)\right] \ge q, \;\; \forall\; u \in \Phi,
\end{equation}
which looks even better than $u^{c}$ in (\ref{CSDef}). However, a definition
using $q>0.5$ is not meaningful for the large majority of stochastic problems
with continuous random variables. Generally speaking, if the sample path
$\omega$ is constructed with continuous random variables, we can have for
$u^{\prime}\neq u^{c}$:
\begin{equation} \label{CSDefp}
\Pr\left[ J(u',\omega)< J(u^c,\omega)\right] = \Pr\left[ J(u',\omega)\le J(u^c,\omega)\right].
\end{equation}
From (\ref{CSDefw}), we have $\Pr\left[  J(u^{\prime},\omega)\leq
J(u^{c},\omega)\right]  \geq q$. Combining it with (\ref{CSDefp}), we have
$\Pr\left[  J(u^{c},\omega)\leq J(u^{\prime},\omega)\right]  \leq1-q$, which
contradicts (\ref{CSDef}) if $q>0.5$. Therefore, even if there might exist
some $u^{\prime}$ that satisfies (\ref{CSDefw}), it will be still the same as
$u^{c}$ defined in (\ref{CSDef}).

The NBA Finals can be used as an example to illustrate the champion solution.
The champion team (the champion solution) will be determined from two teams
(solutions) based on the results in 7 games (sample-paths). The champion
solution is the team (solution) that wins more games (performs better in more
sample-paths). Ideally, if there is an infinite number of games
(sample-paths), then the champion solution is the team with winning ratio of
more than $50\%$.

For cases with more than two solutions, we interpret the champion solution
through the example of presidential elections originally used for Arrow's
Impossibility Theorem in social choice theory \cite{Arrow63}. Imagine we have
three candidates (solutions) \emph{A}, \emph{B} and \emph{C}. Each voter
(sample-path) will rank the three candidates according to his or her own
preference. Now, we randomly pick three voters' preference lists
(sample-paths) as shown in the following table, where $A\succ B$ means
\emph{A} is preferred over \emph{B}. 

\begin{table}[H]
\renewcommand{\arraystretch}{0.6}
\centering
\par
\begin{tabular}{c|ccc}
  \hline  \hline
  \\
  & Voter 1 & Voter 2  & Voter 3\\
  \hline
  \\
Preference  &  $A\succ B \succ C$   &  $B\succ C \succ A$  &  $C\succ B \succ A$\\
  \hline  \hline
\end{tabular}
\end{table}
Based on the the three voters' preferences, we can estimate that
\begin{itemize}
  \item \emph{A}~:~  $\Pr[A\succ B ] = 33\%$, $\Pr[A\succ C ] = 33\%$;
  \item \emph{B}~:~ $\Pr[B\succ A ] = 67\%$, $\Pr[B\succ C ] = 67\%$;
  \item \emph{C}~:~ $\Pr[C\succ A ] = 67\%$, $\Pr[C\succ B ] = 33\%$.
\end{itemize}
Clearly, \emph{B} should be the president (the champion solution) because
\emph{B} gets a higher preference (performs better) than all the other
candidates (solutions) from the majority of voters (sample-paths).

\subsection{Optimality in Expectation vs. Optimality in Probability}

The champion solution favors the winning ratio instead of the winning scale.
That is why we call it \textquotedblleft Champion Solution\textquotedblright.
We can still use the example of NBA Finals. Imagine it was finished in 6 games
and the results are shown in the following table.
\begin{table}[H]
\renewcommand{\arraystretch}{0.6}
\centering
\par
\begin{tabular}{c|cccccc}
  \hline  \hline
  \\
  & Game 1 & Game 2 & Game 3 & Game 4 & Game 5 & Game 6 \\
  \hline
\\
A & 107 & 103 & 84 & 106 & 90 & 98 \\
\\
B & 100 & 97  & 103 & 104 & 101 & 95 \\
  \hline  \hline
\end{tabular}
\end{table}
Team A is the champion (the champion solution) because Team A won more games
than Team B. However, we can also find out that the average score of Team B,
100, is higher than 98, the one of Team A, which implies that Team B is
actually better than Team A in the sense of ``Optimality in Expectation''
commonly adopted in the literature.

Clearly, the champion solution is the best solution in a different sense of
optimality, termed ``Optimality in Probability'' here, which may be a better
optimality sense than the traditional ``Optimality in Expectation'' in some
applications, such as the NBA Finals.

Generally, the champion solution and the traditional optimal solution are not
the same, but they coincide under the following \textquotedblleft
\textbf{Non-singularity Condition}\textquotedblright\ as shown in
\cite{MaoCassJDEDS10a}:
\[
\begin{split}
&  \Pr\left[  J(u',\omega)\leq J(u'',\omega)\right]  \geq0.5 \;\\&  \Longrightarrow \; E\left[  J(u',\omega)\right]  \leq E\left[  J(u'',\omega)\right], \quad \forall u',u''\in \Phi
\end{split}
\]
The interpretation of the Non-singularity Condition is that if $u^{\prime}$ is
more likely better than $u^{\prime\prime}$ (in the sense of resulting in lower
cost), then the expected cost under $u^{\prime}$ will be lower than the one
under $u^{\prime\prime}$. This is consistent with common sense in that any
solution $A$ more likely better than $B$ should result in $A$'s expected
performance being better than $B$'s. Only \textquotedblleft
singularities\textquotedblright\ such as $J(u^{\prime},\omega)\gg
J(u^{\prime\prime},\omega)$ with an unusually low probability for some
$(u^{\prime},u^{\prime\prime})$ can affect the corresponding expectations so
that this condition may be violated. It is straightforward to verify this
Non-singularity Condition for several common cases; for example, consider
$\min_{x}E(x-Y)^{2}$, where $Y$ is a uniform random variable over $[a,b]$. The
optimal solution $(a+b)/2$ satisfies the Non-singularity Condition.

In addition, even though decision makers may prefer \textquotedblleft
optimality in expectation\textquotedblright\ in their applications, the
champion solution still has a very promising performance if the corresponding
problem is not that singular because it can beat all the other solutions with
a probability greater than $0.5$.

\subsection{Sufficient Existence Condition of Champion Solution}

A champion solution may not always exist for a general stochastic optimization
problem. If there are only two feasible solutions, as in the NBA Finals, a
champion solution can be obviously guaranteed. However, this is not the case
even for as few as three feasible solutions. Recalling the example of
presidential elections, what if Voter 3 changes his or her preference as shown
in the following table? 

\begin{table}[H]
\renewcommand{\arraystretch}{0.6}
\centering
\par
\begin{tabular}{c|ccc}
  \hline  \hline
  \\
  & Voter 1 & Voter 2  & Voter 3\\
  \hline
  \\
Preference  &  $A\succ B \succ C$   &  $B\succ C \succ A$  &  $C\succ A \succ B$\\
  \hline  \hline
\end{tabular}
\end{table}
This time we have
\begin{itemize}
  \item \emph{A}~:~  $\Pr[A\succ B ] = 67\%$, $\Pr[A\succ C ] = 33\%$;
  \item \emph{B}~:~ $\Pr[B\succ A ] = 33\%$, $\Pr[B\succ C ] = 67\%$;
  \item \emph{C}~:~ $\Pr[C\succ A ] = 67\%$, $\Pr[C\succ B ] = 33\%$.
\end{itemize}
No candidate can be elected as president (the champion solution) because no
one can be preferred over all the other candidates (solutions) from the
majority of voters (sample-paths); this is in fact the case addressed in
Arrow's paradox \cite{Arrow63}.

In the following, we will establish a sufficient existence condition, which
can be utilized later in the inventory problem considered in the next section.
To accomplish that, we first define the concepts of \textquotedblleft$\omega
$-problem\textquotedblright, \textquotedblleft$\omega$%
-solution\textquotedblright\ and \textquotedblleft$\omega$%
-median\textquotedblright\ for the class of stochastic optimization problems
in (\ref{SOP}). (As these definitions are based on or related to single
sample-path $\omega$, we name their initials as $\omega$-.)

\begin{definition} An $\omega$-problem is the deterministic optimization problem defined over a single sample-path $\omega$, \emph{i.e.},
\[
\min_{u\in \Phi} J(u,\omega).
\]
\end{definition}

\begin{definition} An $\omega$-solution is the optimal solution of  the corresponding $\omega$-problem, \emph{i.e.}, the solution $u^\omega$ such that
\[
u^\omega = \arg\min_{u\in \Phi} J(u,\omega).
\]
\end{definition}

\begin{definition}
The $\omega$-median is the median of the probability distribution of $\omega$-solution $u^\omega$, \emph{i.e.}, the solution $u^m$ such that
\begin{equation}\label{w-Med}
\Pr[u^\omega \le u^m] \ge 0.5  \quad \mbox{and} \quad \Pr[u^\omega \ge u^m] \ge 0.5
\end{equation}
\end{definition}

\noindent\textbf{Remark}: $u^{\omega}$ is a random variable related to
sample-path $\omega$. The two probabilities in (\ref{w-Med}) are the
cumulative distribution function (\emph{cdf}) and complementary cumulative
distribution function (\emph{ccdf}) of $u^{\omega}$ respectively. Both
probabilities can be strictly more than 0.5 at the same time if $u^{\omega}$
is not continuous.

\begin{theorem} \label{CSExist}
If $J(u,\omega)$ is a scalar unimodal function in $u$ for any $\omega$, then the $\omega$-median is a champion solution.
\end{theorem}
\begin{proof}
Since $J(u,\omega)$ is a scalar unimodal function in $u$ for any $\omega$, we have
\begin{equation}\label{LeftSide}J(u',\omega)\le J(u'',\omega), \quad \mbox{for any} \quad u'' < u' <u^\omega;
\end{equation}
and
\begin{equation}\label{RightSide}
J(u',\omega)\le J(u'',\omega), \quad \mbox{for any} \quad u^\omega < u' < u''.
\end{equation}
Assume $u^m$ is the $\omega$-median. For any solution $u>u^m$, we have
\begin{equation}\label{eq201}
\begin{split}
\Pr[J(u^m,\omega) \le J(u,\omega)] & = \Pr[J(u^m,\omega) \le J(u,\omega)|u^\omega \le u^m]\Pr[u^\omega \le u^m]\\ & +\Pr[J(u^m,\omega) \le J(u,\omega)|u^\omega > u^m]\Pr[u^\omega > u^m]
\end{split}
\end{equation}
From (\ref{RightSide}), if $u > u^m$ and $u^m \ge u^\omega$, then $J(u^m,\omega) \le J(u,\omega)$, which implies that
\begin{equation}\label{eq200}
\Pr[J(u^m,\omega) \le J(u,\omega)|u^\omega \le u^m] = 1
\end{equation}
Since $u^m$ is the $\omega$-median, we have $\Pr[u^\omega \le u^m] \ge 0.5$. Combining it with (\ref{eq201}) and (\ref{eq200}), we have
\[
\begin{split}
\Pr[J(u^m,\omega) \le J(u,\omega)]  &\ge 0.5+\Pr[J(u^m,\omega) \le J(u,\omega)|u^\omega > u^m]\Pr[u^\omega > u^m] \\&\ge 0.5
\end{split}
\]

The case of $u<u^m$ can be similarly proved. Therefore, $u^m$ satisfies the definition of champion solution
\[
\Pr[J(u^m,\omega) \le J(u,\omega)] \ge 0.5, \quad \mbox{for any} \;  u \in \Phi.
\]
which implies $u^m$ is a champion solution.
\end{proof}

\subsection{Omega Median Algorithm}

Theorem \ref{CSExist} provides a sufficient existence condition for a champion
solution for a class of simulation-based optimization problems. If it is
satisfied, then a champion solution is guaranteed and can be efficiently
obtained by computing the $\omega$-median. We can efficiently obtain an
estimate of the $\omega$-median using the Omega Median Algorithm (OMA) in
Table \ref{OMA} even though the closed form of the \emph{cdf} and \emph{ccdf}
of $u^{\omega}$ cannot be derived in the class of stochastic optimization
problems in (\ref{SOP}).

\begin{table}[H]
\caption{Omega Median Algorithm}
\label{OMA}
\renewcommand{\arraystretch}{0.8}
\centering
\par%
\begin{tabular}
[l]{p{34pt}p{197pt}}\hline\hline
\textbf{Step 1}: & Randomly generate $M$ sample-paths $\omega^1,...,\omega^M$;\\
& \\
\textbf{Step 2}: & Obtain the $\omega$-solutions, $u^{\omega_i}$, by solving the $\omega$-problems $\min_{u\in \Phi} J(u,\omega_i)$ for $i=1,...,M$;\\
& \\
\textbf{Step 3}: & Find the median solution $\hat u^m$ from $u^{\omega_1},...,u^{\omega_M}$ .\\
\hline\hline
\end{tabular}
\end{table}

The median solution $\hat{u}^{m}$ derived in Step 3 of OMA is an unbiased
estimator of the $\omega$-median. Let $\mathbf{1}(\cdot)$ denote an indicator
function and
\[
\begin{split}
&G_M(u) \equiv \frac{1}{M}\sum\nolimits_{j=1}^M \mathbf{1} (u^{\omega_j} \le u);\\& \bar G_M(u) \equiv \frac{1}{M}\sum\nolimits_{j=1}^M \mathbf{1} (u^{\omega_j} \ge u).
\end{split}
\]
Then, $G_{M}(u)$ and $\bar{G}_{M}(u)$ are the estimates of the \emph{cdf} and
\emph{ccdf} of $u^{\omega}$ respectively. It can be easily verified that the
median solution $\hat{u}^{m}$ is the solution that satisfies
\[
G_M(\hat u^m)\ge 0.5 \quad \mbox{and}\quad \bar G_M(\hat u^m)\ge 0.5~.
\]
For any given $u$, based on the strong law of large numbers, $G_{M}(u)$ and
$\bar{G}_{M}(u)$ converge to $\Pr[u^{\omega}\leq u]$ and $\Pr[u^{\omega}\geq
u]$ respectively \emph{w.p.}1 (with probability 1) as $M\rightarrow+\infty$.
Thus, $\hat{u}^{m}$ also converges to the $\omega$-median $u^{m}$ \emph{w.p.}1
as $M\rightarrow+\infty$.

Furthermore, $\hat u^{m}$ can approach the $\omega$-median $u^{m}$
exponentially fast as $M$ increases as shown in Theorems \ref{DiscreteConv}
and \ref{ContinuousConv} below, which enables us to estimate the $\omega
$-median with a smaller number $M$ of sample paths.

\begin{theorem} \label{DiscreteConv} If $\Pr(u^\omega = u^m)>0$, then there always exists some constant $C$ such that
\[
\Pr[\hat u^m = u^m] \ge 1 - 2e^{-CM}
\]
\end{theorem}
\begin{proof}
Without loss of generality, assume $\Pr(u^\omega = u^m)=c>0$, $\Pr(u^\omega < u^m)=p_1$ and $\Pr(u^\omega > u^m)=p_2$. From the definition of $\omega$-median, we have $p_1 + c \ge 0.5$ and $p_2 +c \ge 0.5$. Combining it with $p_1 + c + p_2 = 1$ and $c>0$, we have
\[
p_1<0.5, \quad p_2<0.5.
\]
The event $[\hat u^m = u^m]$ is equivalent to the event $[G_M(u^m)\ge 0.5 \mbox{ and } \bar G_M(u^m)\ge 0.5]$, which can be further equivalently reduced to $[L_M(\hat u^m)< 0.5 \mbox{ and } \bar L_M(\hat u^m)< 0.5]$, where
\[
L_M(u) = \frac{1}{M}\sum_{j=1}^M \mathbf{1} (u^{\omega_j} < u), \quad \bar L_M(u) = \frac{1}{M}\sum_{j=1}^M \mathbf{1} (u^{\omega_j} > u).
\]
Therefore, we have
\begin{equation} \label{eqevents}
\begin{split}
\Pr[\hat u^m = u^m]& = \Pr[L_M(\hat u^m)< 0.5 \mbox{ and } \bar L_M(\hat u^m)< 0.5]\\ &= 1 - Pr[L_M(u^m)> 0.5 \mbox{\ or\ } \bar L_M(u^m)> 0.5] \\
& = 1 - \big(\Pr[L_M(u^m)> 0.5] + Pr[\bar L_M(u^m) > 0.5]\big)
\end{split}
\end{equation}
Clearly, $\mathbf{1} (u^{\omega_j} < u^m), j= 1,...,M$ are \emph{i.i.d.} 0-1 random variables and $E[\mathbf{1} (u^{\omega_j} < u^m) ] = p_1$. Then based on Chernoff-Hoeffding Theorem \cite{Hoeffding63}, we have for any $\epsilon>0$
\[
\Pr[L_M(u^m)\ge p_1 + \epsilon] \le e^{-D(p_1+\epsilon||p_1)M}
\]
where $D(x||y)=x\log\frac{x}{y}+(1-x)\log\frac{1-x}{1-y}$. Similarly, we can also have
\[
\Pr[\bar L_M(u^m)\ge p_2 + \epsilon] \le e^{-D(p_2+\epsilon||p_2)M}
\]
Combining the two inequalities above with $p_1<0.5$ and $p_2<0.5$, we can further have
\[
\Pr[L_M(u^m)> 0.5]  \le \Pr[L_M(u^m) \ge 0.5] \le e^{-D(0.5||p_1)M}
\]
\[
\Pr[\bar L_M(u^m)> 0.5]  \le \Pr[\bar L_M(u^m) \ge 0.5] \le e^{-D(0.5||p_2)M}
\]
Combining them with (\ref{eqevents}), we can finally have
\[
\begin{split}
\Pr[\hat u^m = u^m]&\ge 1 - e^{-D(0.5||p_1)M} - e^{-D(0.5||p_2)M}\ge 1 - 2e^{-CM}
\end{split}
\]
where $C = \min\big(D(0.5||p_1),D(0.5||p_2)\big)$
\end{proof}

\begin{theorem} \label{ContinuousConv} If $\Pr(u^\omega = u^m)=0$, then for any $\epsilon>0$, there always exists $C>0$ such that
\[
\begin{split}
&\Pr\big[\;|G_M(u^m) - 0.5| < \epsilon\big] \ge 1 - 2e^{-CM}, \\
&\Pr\big[\;|\bar G_M(u^m) - 0.5| < \epsilon\big]  \ge 1 - 2 e^{-CM} .
\end{split}
\]
\end{theorem}
\begin{proof} From $\Pr(u^\omega = u^m)=0$ and the definition of $u^m$, we have
\[
\Pr[u^\omega \le u^m] = 1- \Pr[u^\omega \ge u^m] = 0.5
\]
which implies that
\[
E\big[G_M(u^m)\big] = 0.5
\]
Since $\mathbf{1} (u^{\omega_j} \le u^m), j= 1,...,M$ are \emph{i.i.d.} 0-1 random variables and $E[\mathbf{1} (u^{\omega_j} < u^m) ] = 0.5$, based on Chernoff-Hoeffding Theorem \cite{Hoeffding63}, we have for any $\epsilon>0$
\[
\begin{split}
&\Pr[G_M(u^m) \ge 0.5 + \epsilon] \le e^{-D(0.5+\epsilon||0.5)M} \mbox{ \ and \ } \\ & \Pr[G_M(u^m) \le 0.5 - \epsilon] \le e^{-D(0.5-\epsilon||0.5)M}
\end{split}
\]
where $D(x||y)=x\log\frac{x}{y}+(1-x)\log\frac{1-x}{1-y}$. Therefore, we have
\[
\begin{split}\Pr\big[\;|G_M(u^m) - 0.5| < \epsilon\big]& = 1 - \Pr[G_M(u^m) \ge 0.5 + \epsilon] - \Pr[G_M(u^m) \le 0.5 - \epsilon]\\
&\ge 1 - e^{-D(0.5+\epsilon||0.5)M}  - e^{-D(0.5-\epsilon||0.5)M} \\
& \ge 1 - 2 e^{-CM} .
\end{split}
\]
where $C = \min\big(D(0.5+\epsilon||0.5),D(0.5-\epsilon||0.5)\big)$.

It can be similarly proved that
\[
\begin{split}
\Pr\big[\;|\bar G_M(u^m) - 0.5| < \epsilon\big]&  \ge 1 - 2 e^{-CM} .
\end{split}
\]
\end{proof}

Theorem \ref{DiscreteConv} corresponds to the case that $u$ is discrete and
Theorem \ref{ContinuousConv} is mainly for the case that $u$ is continuous.
Theorem \ref{DiscreteConv} has a stronger sense of convergence than Theorem
\ref{ContinuousConv}, which implies that $\hat u^{m}$ converges faster in
discrete cases than in continuous ones.

\section{An Example: Inventory Control with Nonstationary Demand} \label{OICP}

To illustrate and interpret the use of the Omega Median Algorithm, we consider
an on-line periodic review inventory control problem with nonstationary demand
as depicted in Figure \ref{Process} as a discrete event system (DES), in which
fixed setup cost and full backlogging are adopted. The following notation will
be used in the rest of the paper:
\begin{itemize}
  \item $x_i=$ Inventory level in period $i$;
  \item $d_i=$ Demand in period $i$;
  \item $u_i=$ Order quantity in period $i$;
  \item $h=$ Holding cost rate for inventory;
  \item $p=$ Penalty cost rate for backlog;
  \item $K=$ Fixed setup cost per order;
  \item $\delta(u_i)=\left\{ {\begin{array}{*{20}{c}}
~1~~~&{{u_i} > 0}\\
~0~~~&{{u_i} = 0}
\end{array}} \right. .$
\end{itemize}
The one-period demand $d_{i}$ is nonstationary, \emph{i.e.}, its corresponding
probability distribution is arbitrary and allowed to vary and correlate over
periods $i$.

\begin{figure}[H]
\centerline{\includegraphics[width=0.7\textwidth]{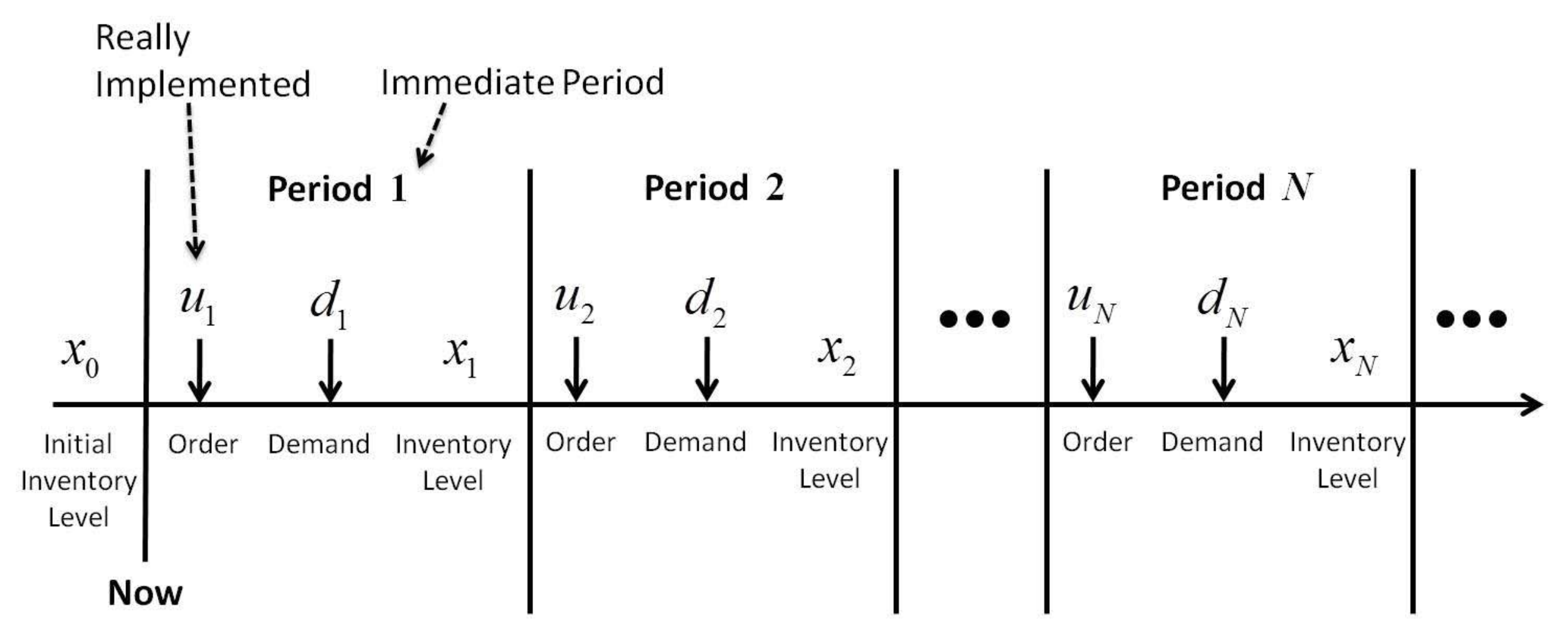}}\caption{On-line Inventory Control Process}
\label{Process}
\end{figure}

An ordering event may be triggered at the beginning of a period, namely, an
order of $u_{i}$ items may be placed in period $i$. A fixed setup cost $K$
will be triggered if $u_{i}>0$. The inventory level $x_{i}$ is counted after
the one-period demand $d_{i}$, \emph{i.e.}, $x_{i}=x_{i-1}+u_{i}-d_{i}$, which
results in the maintenance cost of period $i$ (either holding or shortage
cost) defined below,
\begin{equation}\label{Hfunction}
H(x_i) = h\cdot\max(x_i,0) + p\cdot\max(-x_i,0).
\end{equation}
The average operating cost in each period, including both maintenance cost and
setup cost, determines the system performance.

The static $(s,S)$ policy is an optimal policy for the cases with stationary
demands using optimality in expectation. Once the two thresholds $(s,S)$ are
optimally determined, the corresponding optimal ordering quantity can be
simply derived as $u_{i}=S-x_{i-1}$ if $x_{i-1}\le s$ and $u_{i}=0$ otherwise.
However, the static $(s,S)$ policy is not optimal for nonstationary demands
\cite{Axsater06}: the optimal order decisions cannot be simply derived by
optimizing the two thresholds $(s,S)$, as in the algorithm in
\cite{ZhengFed91} that requires integer-valued and \emph{i.i.d.} (independent
and identical distributed) one-period demands. Some efforts have been made
towards the nonstationary inventory control problem with fixed setup cost
\cite{Askin81,Bollapragada99}. A heuristic similar to Silver-Meal heuristics
is proposed in \cite{Askin81} and requires to explicitly compute the
probability distributions of cumulative demands, which is not plausible for
general nonstationary demands with complicated patterns. In
\cite{Bollapragada99}, nonstationary demands are approximated by averaging
demands over periods and then a stationary policy is computed by utilizing the
algorithm in \cite{ZhengFed91}, which will be benchmarked against the proposed
Omega Median Algorithm in the numerical results section below.

Although general simulation-based methods can still be utilized to determine
the best order decision using optimality in expectation, it is computationally
intensive or even intractable as analyzed in Section \ref{Complexity}.
Instead, we pursue the best solution in the sense of optimality in
probability, namely, the \textquotedblleft Champion Solution\textquotedblright%
, which is a very good alternative when facing a nonstationary environment.

In the on-line inventory control process depicted in Fig \ref{Process}, we
make an order decision at the beginning of each period. The rolling horizon
method can be applied, in which we look ahead $N$ periods and the actual
performance over a specific $N$-period sample path $\omega=\{d_{1}%
,d_{2},...,d_{N}\}$ can be defined as the total cost:
\begin{equation}\label{JN}
\begin{split}
J_N(&u_1,u_2,...,u_N,\omega) = \sum\nolimits_{i=1}^N \big(H(x_i) + K\cdot \delta(u_i)  \big)  \\
& \quad {s.t.} \;\; x_i = x_{i-1} - d_i + u_i, \; i=1,...,N.
\end{split}
\end{equation}
where $H(x_{i})+K\cdot\delta(u_{i})$ is the operating cost in period $i$,
including maintenance cost and setup cost.

Since only the immediate-period order decision, $u_{1}$, is required each
time, we will focus on $u_{1}$ and optimally determine $u_{2},...,u_{N}$ based
on the choice of $u_{1}$. Then, the actual performance over a specific
$N$-period sample path $\omega$ becomes solely associated with $u_{1}$ as
follows:
\begin{equation}\label{JN1}
\begin{split}
J_N(&u_1,\omega) = \big(H(x_1) + K\cdot \delta(u_1)\big)\\&  + \min_{u_2,...,u_N} \sum\nolimits_{i=2}^N \big(H(x_i) + K\cdot \delta(u_i)  \big)  \\
& {s.t.} \;\; x_i = x_{i-1} - d_i + u_i, \; i=1,...,N.
\end{split}
\end{equation}
In the ideal case of looking ahead for an infinite horizon, the actual
performance over a specific sample path $\omega$ can be formulated as the
infinite-horizon average cost:
\begin{equation}\label{JInf}
\begin{split}
J(&u_1,\omega)\equiv \lim_{N\rightarrow +\infty} \frac{1}{N} \big\{J_N(u_1,\omega)\big\}
\end{split}
\end{equation}
We aim at the champion solution using the actual performance function in
(\ref{JInf}).

\subsection{Existence of Champion Solution}

The inventory control problem can be solved by sequentially answering the two
questions below.
\begin{table}[H]
\renewcommand{\arraystretch}{1}
\centering
\par
\begin{tabular}
[l]{p{50pt}p{190pt}}
\textbf{Question 1}: & Whether to order (Yes or No);\\
\\
\textbf{Question 2}: & How many items to order if ``Yes'' to Question 1. \\
\end{tabular}
\end{table}

Since Question 1 has only two options, its champion solution can be guaranteed
and easily obtained as follows,
\[\left\{ {\begin{array}{*{20}{c}}
   {\mbox{Yes~}} & {\mbox{\ if\ } \Pr[u_1^\omega > 0]\ge 50\%}  \\
   {\mbox{No~}} & {\mbox{otherwise.}}  \\
\end{array}} \right.\]
where $u_{1}^{\omega}$ is the $\omega$-solution of minimizing $J(u_{1}%
,\omega)$ in (\ref{JInf}) and $\Pr[u_{1}^{\omega}>0]$ is the probability to
place a positive order.

Question 2 is conditioned on \textquotedblleft Yes\textquotedblright\ to
Question 1, which implies that $u_{1}>0$ in Question 2. In the following, we
will verify the existence of a champion solution for $u_{1}>0$ with the help
of the lemma below.

\begin{lemma} \label{lem3} $J_N(u_1,\omega)$ in (\ref{JN1}) is $K$-convex in $u_1$ for $u_1>0$.
\end{lemma}
\begin{proof}
It can be easy to prove that $L_N(x_1,\omega)$ is $K$-convex in $x_1$ using a similar way as shown in Section 4.2 in \cite{Bertsekas00}. Combining it with $x_1 = u_1+x_0-d_1$, $L_N(u_1+x_0-d_1,\omega)$ is also $K$-convex in $u_1$.

From the definition of $H(x)$ in (\ref{Hfunction}), $H(x_1)$ is convex in $x_1$, which implies $H(u_1+x_0-d_1)$ is also convex in $u_1$.

Recalling the definition of $J_N(u_1,\omega)$ in (\ref{JN1}). From $u_1>0$, we have
\[
J_N(u_1,\omega) = H(u_1+x_0-d_1) + K + L_N(u_1+x_0-d_1,\omega)
\]
Combining it with the fact that $H(u_1+x_0-d_1)$ is convex in $u_1$ and $L_N(u_1+x_0-d_1,\omega)$ is $K$-convex in $u_1$, we have  $J_N(u_1,\omega)$ is $K$-convex in $u_1$ for $u_1>0$.
\end{proof}

Based on Lemma \ref{lem3} and the definition of $J(u_{1},\omega)$ in
(\ref{JInf}), we prove the following theorem.

\begin{theorem}\label{JUniModal} $J(u_1,\omega)$ is convex in $u_1$ for $u_1>0$.
\end{theorem}
\begin{proof} From Lemma \ref{lem3}, $J_N(u_1,\omega)$ is $K$-convex
in $u_1$ for $u_1>0$, that is, it satisfies that for any $0< u_1<u_1'<u_1''$
\[
\begin{split}
K + J_N(u_1'',\omega) &\ge J_N(u_1',\omega) +
(\frac{u_1''-u_1'}{u_1'-u_1})(J_N(u_1',\omega)-J_N(u_1,\omega)).
\end{split}
\]
Then we apply limit operator at both sides and can have
\[
\begin{split}
\lim_{N\rightarrow +\infty}& \frac{K + J_N(u_1'',\omega) }{N}\ge  \lim_{N\rightarrow +\infty} \frac{J_N(u_1',\omega)}{N} +
(\frac{u_1''-u_1'}{u_1'-u_1})\lim_{N\rightarrow +\infty}\frac{(J_N(u_1',\omega)-J_N(u_1,\omega))}{N}
\end{split}
\]
which implies that for any $0< u_1<u_1'<u_1''$,
\[
\begin{split}
J(u_1'',\omega) \ge J(u_1',\omega) +
(\frac{u_1''-u_1'}{u_1'-u_1})(J(u_1',\omega)-J(u_1,\omega)).
\end{split}
\]
The inequality above is equivalent to the definition of convex function, that is, $J(u_1,\omega)$ is convex in $u_1$ for $u_1>0$.
\end{proof}

Theorem \ref{JUniModal} implies that $J(u_{1},\omega)$ is unimodal for
$u_{1}>0$, which satisfies the sufficient existence condition identified in
Theorem \ref{CSExist}. Therefore, a champion solution can be guaranteed to
address Question 2 and can be obtained using OMA.

\subsection{Implementation of OMA}

Although $d_{i}$, $i=1,2,\ldots$, is nonstationary, we can still estimate
their probability distributions based on the most recently updated
information. Sample paths can then be randomly generated in Step 1 of OMA
using these estimates.

Step 2 of OMA determines the major portion of its computational complexity,
which can be largely reduced if we manage to find an efficient algorithm to
solve the corresponding $\omega$-problems. In the context of this inventory
control problem, the $\omega$-problem is to find the $\omega$-solution
$u_{1}^{\omega}$ of minimizing $J(u_{1},\omega)$ in (\ref{JInf}). This
$\omega$-solution $u_{1}^{\omega}$ can be well approximated by minimizing
$J_{N}(u_{1},\omega)$ in (\ref{JN1}) with a large enough $N$. Furthermore, it
can be easily verified that, if $u_{1}^{\ast},...u_{N}^{\ast}$ can minimize
$J_{N}(u_{1},...,u_{N},\omega)$ in (\ref{JN}), then $u_{1}^{\ast}$ can also
minimize $J_{N}(u_{1},\omega)$ in (\ref{JN1}). Therefore, we can finally
obtain the $\omega$-solution $u_{1}^{\omega}$ by minimizing $J_{N}%
(u_{1},...,u_{N},\omega)$ in (\ref{JN}) with a sufficiently large $N$.

The problem of minimizing $J_{N}(u_{1},...,u_{N},\omega)$ in (\ref{JN}) is
closely related to the following problem, which is a dynamic lot-sizing
problem with backlogging as defined in the literature \cite{FedergruenTzur93}.
\begin{equation}\label{F-SSPP}
\begin{split}
&  \min_{u_1,...,u_N} \sum\nolimits_{i=1}^N\big\{ H(x_i) + K\cdot \delta(u_i)\big\}\\
& {s.t.} \;\; x_i = x_{i-1} - d_i + u_i, \; i=1,...,N;\\
&\quad \;\; \sum\nolimits_{i=1}^N u_i + x_0 = \sum\nolimits_{i=1}^N d_i.
\end{split}
\end{equation}
The only difference between the two problems results from the second
constraint, which can be interpreted as the condition of \textquotedblleft
zero inventory at last\textquotedblright. Since profits earned from sales are
not included in the objective, it would never be optimal to place a new order
at the last period which would mostly end up with a negative inventory level.
The terminal effect of \textquotedblleft ordering nothing at
last\textquotedblright\ and \textquotedblleft ending with negative
inventory\textquotedblright\ are quite undesirable. Solving the problem in
(\ref{F-SSPP}) instead with the extra second constraint can be very helpful in
approximating the $\omega$-solution when using a relatively small $N$. Since
the problem in (\ref{F-SSPP}) has been well studied in \cite{FedergruenTzur93},
we can efficiently solve each $\omega$-problem with complexity $O(N\log N)$
for general cases.

The remaining Step 3 of OMA can be trivially fulfilled once we have $M$
$\omega$-solutions.

\subsection{Complexity Analysis}\label{Complexity}

Clearly, the complexities of Step 1 and 3 of OMA are $O(MN)$ and $O(M)$
respectively. With the help of the algorithm in \cite{FedergruenTzur93}, the
complexity of Step 2 is $O(M\cdot N\log N)$. Thus, we can finally efficiently
obtain a champion solution of the nonstationary inventory control problem in
complexity $O(M\cdot N\log N)$ by applying OMA.

If we try a general simulation-based optimization method using optimality in
expectation, then we need to solve the following stochastic optimization
problem (\ref{JN-EXP}) at each decision point,
\begin{equation}\label{JN-EXP}
\begin{split}
&\min_{u_1} \;\bar J_N(u_1) =E\bigg\{ \big(H(x_1) + K\cdot \delta(u_1)\big) \\ & + \min_{\mu_2,...,\mu_N} E \Big\{\sum\nolimits_{i=2}^N \big(H(x_i) + K\cdot \delta(u_i)  \big)  \Big\}\bigg\} \\
& \quad {s.t.} \;\; x_i = x_{i-1} - d_i + u_i, \; i=1,...,N;\\
& \quad \quad\;\;\; u_i = \mu_i(x_{i-1}) , \; i=2,...,N.
\end{split}
\end{equation}
where $\mu_{i}(\cdot)$ is the feedback control policy to determine $u_{i}$
based on the state $x_{i-1}$. Clearly, even for a given $u_{1}$, computing
$\bar{J}_{N}(u_{1})$ is a notoriously hard dynamic programming problem.
Although a heuristic termed \textquotedblleft Hindsight
Optimization\textquotedblright\ \cite{ChongGivanChang00} can be employed to
approximate the second term in the objective of (\ref{JN-EXP}) as the expected
hindsight-optimal value below,
\[
\begin{split}
E \left\{\min_{u_2,...,u_N} \sum\nolimits_{i=2}^N \Big(H(x_i) + K\cdot \delta(u_i)  \Big)  \right\},
\end{split}
\]
still requires a complexity of $O(M\cdot N\log N)$ to assess a specific choice
of $u_{1}$. Moreover, it needs to go through a search process to get a near
optimal $u_{1}$. If there are a total og $I$ solutions explored in the
process, then the total computational complexity is $O(M\cdot I\cdot N\log
N)$, which is an order of magnitude higher than that of OMA.

\section{Numerical Results}

We illustrate the performance of OMA through a numerical example. The
following parameters are identical to those used in \cite{Zheng91},
\begin{itemize}
  \item Fixed Setup Cost $K=64$;
  \item Holding Cost Rate $h=1$;
  \item Penalty Cost Rate $p=9$.
\end{itemize}
A case of nonstationary demands is considered, in which demand in each period
is Poisson distributed and may has a different mean value $\mu_{i}$. The mean
value $\mu_{i}$ will be randomly picked from a set of numbers between 10 an 75
in increments of 5, that is, $\{10,15,20,...,70,75\}$.

\subsection{$\omega$-median Approximation}

An example of estimating the $\omega$-median is shown in Figure
\ref{wMedEstimation}, in which $M=200$ sample-paths are generated. The
$\omega$-solutions are obtained by solving $200$ corresponding $\omega
$-problems through the algorithm in \cite{FedergruenTzur93}.

The solid line in Figure \ref{wMedEstimation} is the \emph{cdf} function of
the $\omega$-solution constructed based on these sample-paths. The estimate of
the $\omega$-median is $u^{m}=78$, which is indicated through the dashed line.

\begin{figure}[H]
\centerline{\includegraphics[width=0.5\textwidth]{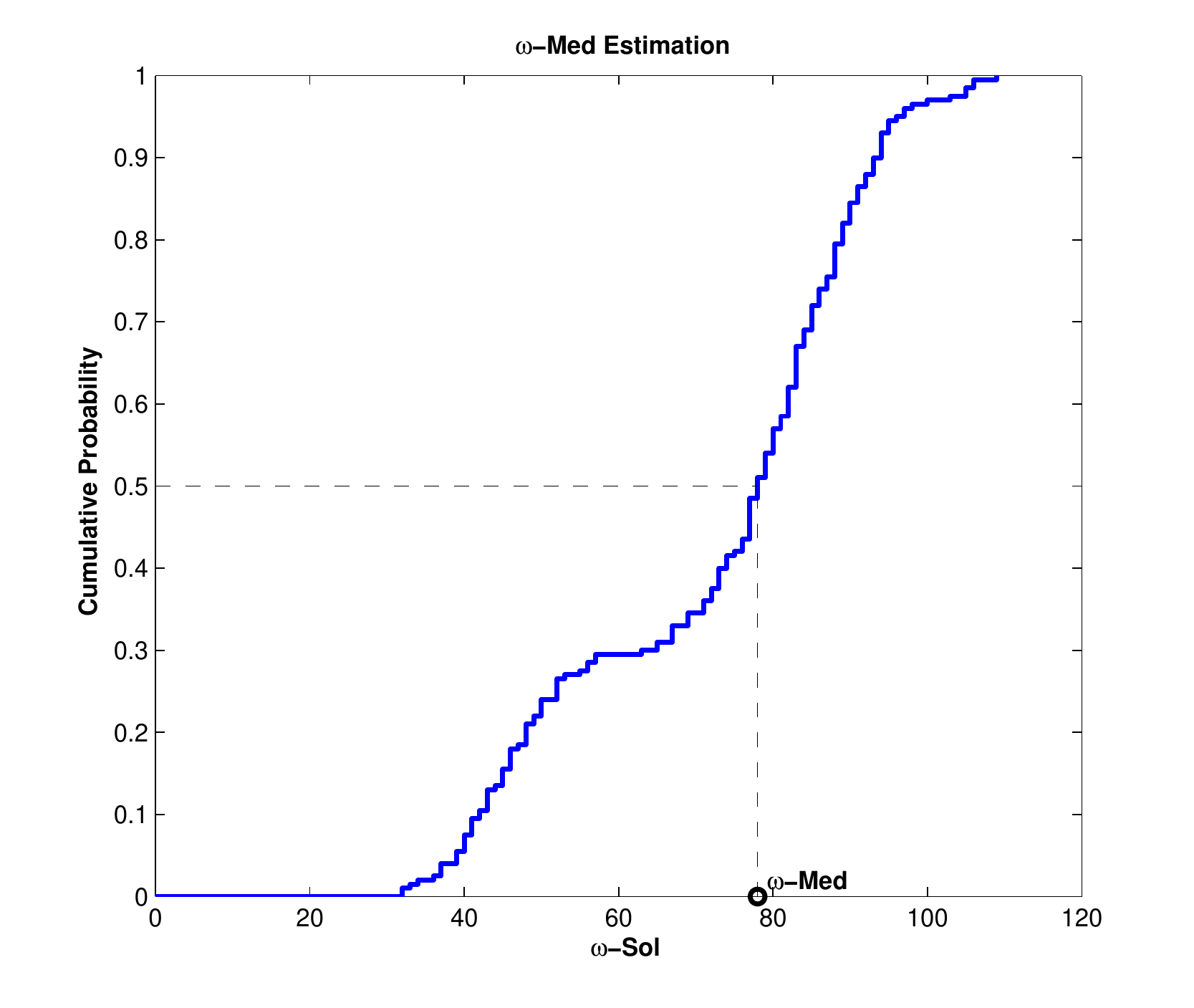}}\caption{$\omega$-median Approximation}
\label{wMedEstimation}
\end{figure}

\subsection{Convergence of $\omega$-median in $M$}

The convergence of the $\omega$-median in the number of sample-paths $M$ is
shown in Figure \ref{wMedConvergence}, in which $M$ varies from 10 to 1000 in
increments of 10. It can be seen that the estimate of the $\omega$-median
quickly converges within $100$ replications, which supports the result in
Theorem \ref{DiscreteConv}.

\begin{figure}[H]
\centerline{\includegraphics[width=0.5\textwidth]{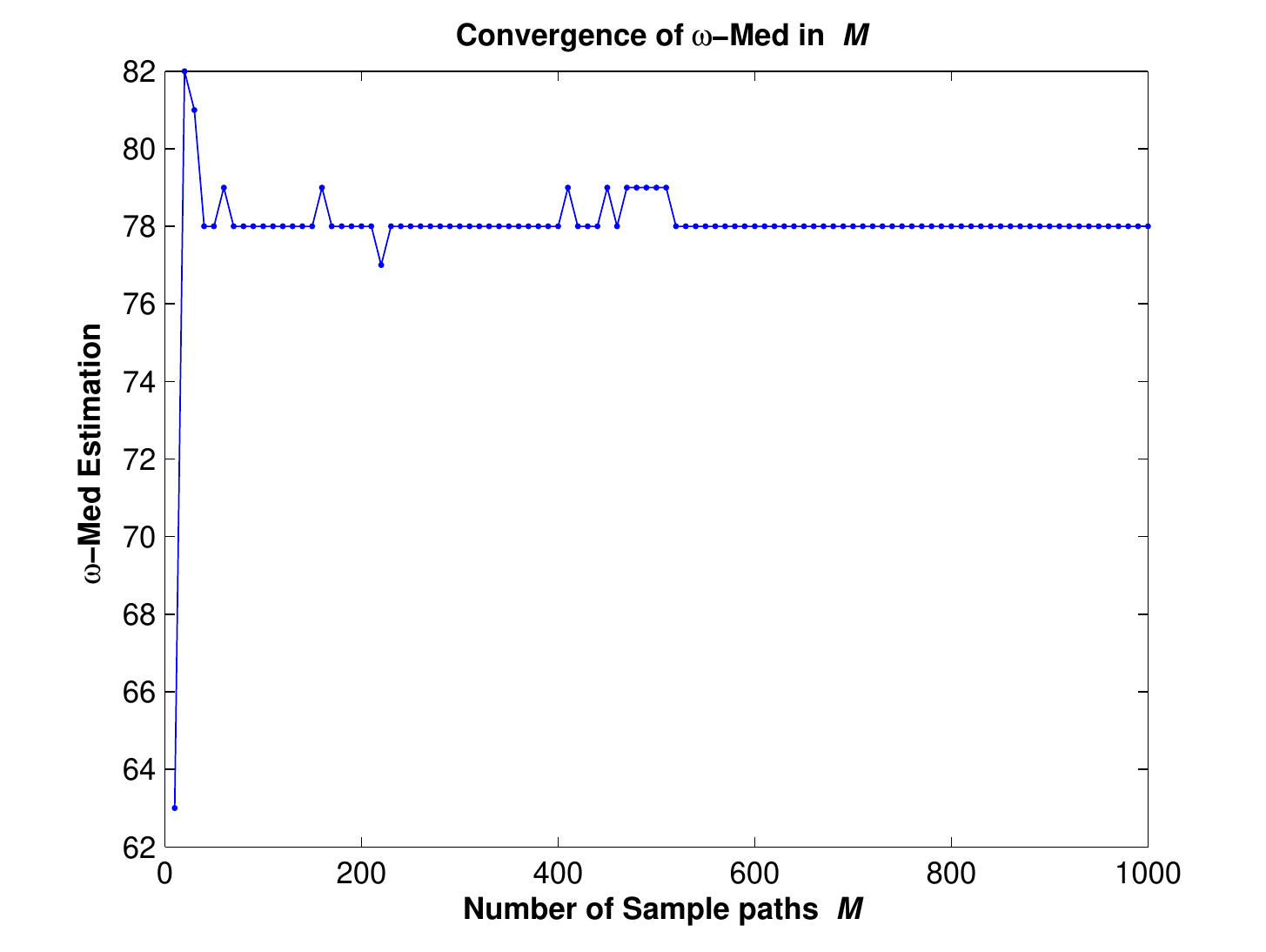}}
\caption{Convergence of $\omega$-median in $M$}
\label{wMedConvergence}
\end{figure}

\subsection{Stationary Cases: ``Optimality in Expectation'' vs. ``Optimality in Probability''}

The optimal static policies $(s^{\ast},S^{\ast})$ have been exactly derived by
using the algorithm in \cite{ZhengFed91} for stationary cases with different
$\mu$. This provides an opportunity to benchmark the performance of the
champion solution against the optimal policy $(s^{\ast},S^{\ast})$, the best
solution in the sense of \textquotedblleft optimality in
expectation\textquotedblright.

We set $\mu=20$ in the following experiment, in which 20 instances with $N=50$
periods are randomly generated, and we compare the two methods below:

\begin{enumerate}
  \item Method \textbf{SS}: Order decisions are directly obtained according to the optimal static policy  $(s^*=14,S^*=62)$ as obtained in \cite{ZhengFed91};
  \item Method  \textbf{CS}: Order decisions are obtained by using the $\omega$-median approximation with $M=100$ sample paths at the beginning of each period, namely, the estimates of champion solutions.
\end{enumerate}

The performance comparison results are listed in Table \ref{tab5}, in which
the first column is the instance index, the second column is the cost $C_{ss}$
of using method \textbf{SS}, the third column is the cost $C_{cs}$ of using
method \textbf{CS}, the fourth column is the difference between the two costs
and the fifth column is the fractional improvement defined as $\frac
{(C_{cs}-C_{ss})}{C_{ss}}$.

\begin{table}
[H]
\caption{Stationary Demands}\label{tab5}
\renewcommand{\arraystretch}{0.6}
\centering
\par
\begin{tabular}{>{\centering\arraybackslash}m{25pt}| >{\centering\arraybackslash}m{50pt} | >{\centering\arraybackslash}m{50pt} | >{\centering\arraybackslash}m{50pt} | >{\centering\arraybackslash}m{50pt} } 
\hline\hline
 &  Cost of \textbf{SS}  & Cost of \textbf{CS} & Difference & Improvement\\
  & $C_{ss}$ & $C_{cs}$ & $C_{ss}-C_{cs}$ & $\frac{(C_{cs}-C_{ss})}{C_{ss}}$\\
 \hline
&	&	& & \\
1&	2401&	2439	&-38	&-1.58\%\\
2& 2710 &	2590	&120&	4.43\%\\
3&	2525 &	2561&	-36&	-1.43\%\\
4& 2612	&2574&	38	&1.45\%\\
5& 2450	&2700	&-250	&-10.20\%\\
6&	2390	&2724	&-334	&-13.97\%\\
7&	2401	&2552	&-151	&-6.29\%\\
8&	2711	&2516	&195	&7.19\%\\
9&	2410	&2670&	-260&	-10.79\%\\
10&	2598	&2454&	144&	5.54\%\\
11	& 2563	&2559	&4	&0.16\%\\
12	& 2441	&2570	&-129&	-5.28\%\\
13	& 2530	&2469&	61	&2.41\%\\
14	& 2419	&2446	&-27	&-1.12\%\\
15	& 2571	&2488&	83	&3.23\%\\
16	& 2365	&2599	&-234&	-9.89\%\\
17	& 2622	&2542	&80	&3.05\%\\
18&	2672	&2502	&170&	6.36\%\\
19&	2480	&2372	&108&	4.35\%\\
20	& 2543	&2608	&-65	&-2.56\%\\
\hline
&	&	& & \\
Mean&	2520.7	&2546.75	&-26.05	&-1.03\%\\
 \hline  \hline
\end{tabular}
\end{table}

From Table \ref{tab5}, the average operating cost of \textbf{SS} is slightly less than the one of \textbf{CS}, which confirms that the order decisions based on the optimal policy $(s^*,S^*)$ are truly the best in the sense of optimality in expectation.

We can also observe that the order decisions based on the estimated champion solutions perform better than the ones based on the optimal policy $(s^*,S^*)$ in 10 instances, \emph{i.e.}, instances 2, 4, 8, 10, 11, 13, 15, 17, 18 and 19. \textbf{CS} has a winning ratio of $50\%$ against \textbf{SS} based on these 20 instances, which implies that the estimated champion solutions perform as well as the exact optimal policy in the sense of optimality in probability in this numerical experiment. Besides, the estimated champion solutions are not the exact champion solutions and we can further improve the performance by increasing the sample size $M$.

Even though decision makers may prefer the sense of optimality in expectation, the estimated champion solutions are near-optimal, since their corresponding average cost is only $1.03\%$ worse than the one of the optimal policy in expectation.

\subsection{Nonstationary Cases}

In the following experiments of nonstationary cases, we set different $\mu_i$ for each period, which are randomly selected from the values listed in $\{10, 15, 20, ..., 70, 75\}$.

\begin{table}[H]
\caption{Nonstationary Demands}\label{tab6}
\renewcommand{\arraystretch}{0.6}
\centering
\par
\begin{tabular}{>{\centering\arraybackslash}m{25pt}| >{\centering\arraybackslash}m{50pt} | >{\centering\arraybackslash}m{50pt} | >{\centering\arraybackslash}m{50pt} | >{\centering\arraybackslash}m{50pt} } 
\hline\hline
 & Cost of SS & Cost of CS & Difference & Improvement\\
  & $C_{ss}$ & $C_{cs}$ & $C_{ss}-C_{cs}$ & $\frac{(C_{ss}-C_{cs})}{C_{ss}}$\\
 \hline
 &	&	& & \\
1&	3506&	2908	&598	&17.06\%\\
2&	3642	&2938	&704	&19.33\%\\
3&	3467	&3073	&394	&11.36\%\\
4&	3611	&3022	&589	&16.31\%\\
5&	3540	&3004	&536	&15.14\%\\
6&	3519	&3092	&427	&12.13\%\\
7&	3516	&3033	&483	&13.74\%\\
8&	3782	&3096	&686	&18.14\%\\
9&	3440	&2989	&451	&13.11\%\\
10&	3567	&2907	&660	&18.50\%\\
11	&3846	&2992	&854	&22.20\%\\
12	&3251	&2918	&333	&10.24\%\\
13	&3388	&2750	&638	&18.83\%\\
14	&2990	&2807	&183	&6.12\%\\
15	&3434	&2868	&566	&16.48\%\\
16	&3633	&3167	&466	&12.83\%\\
17	&3643	&2984	&659	&18.09\%\\
18	&3535	&3038	&497	&14.06\%\\
19	&3456	&3192	&264	&7.64\%\\
20	&3251	&3071	&180	&5.54\%\\
\hline
&	&	& & \\
Mean&	3500.85&	2992.45&	508.4&	14.52\%\\
 \hline  \hline
\end{tabular}
\end{table}

We again generate 20 instances with $N = 50$ periods and compare two methods below:

\begin{enumerate}
  \item Method \textbf{SS}: Order decisions are directly obtained according to a heuristic nonstationary policy $(s_i,S_i)$ for each period $i$.  A common heuristic method is to determine $(s_i,S_i)$ according to $\mu_i$ in the corresponding period $i$ as if demands are stationary with the mean value of $\mu_i$. For example, if $\mu_1 = 15, \mu_2 = 30, \mu_3 = 20, ...$, then we can look up the table obtained in \cite{ZhengFed91} to find their corresponding optimal values, choose $(s_1 = 10,S_1=49)$, $(s_2 = 23, S_2 =66)$, $(s_3=14,S_3=62), ...,$ to apply in period $1, 2, 3, ...,$ respectively. Clearly, this heuristic $(s_i,S_i)$ policy is not optimal for the nonstationary case.
  \item Method  \textbf{CS}: Order decisions are still obtained by using the $\omega$-median approximation with $M=100$ sample paths at the beginning of each period, namely, the estimates of champion solutions.
\end{enumerate}

The performance comparison results are listed in Table \ref{tab6} that shares a similar organization with Table \ref{tab5}. It can be easily seen that the estimated champion solutions result in a $14.52\%$ lower average cost and perform better than the heuristic $(s_i,S_i)$ policy in all 20 instances.

\section{Conclusion}

An alternate optimality sense, optimality in probability, is proposed in this paper. The best solution using optimality in probability is termed a ``Champion Solution'' whose actual performance is more likely better than that of any other solution. A sufficient existence condition for the champion solution is proved for a class of simulation-based optimization problems. A highly efficient method, the Omega Median Algorithm (OMA), is developed to compute the champion solution without iteratively exploring better solutions based on sample average approximations. OMA can reduce the computational complexity by orders of magnitude compared to general simulation-based optimization methods using optimality in expectation.

The champion solution becomes particularly meaningful when facing a nonstationary environment. As shown in the example of inventory control with nonstationary demand, the solution using optimality in expectation is not necessarily optimal and is computationally intractable in a dynamic environment. The champion solution is a good alternative and computationally promising. Its corresponding solution algorithm, OMA, can fully utilize the efficiency of those well-developed off-line algorithms to further facilitate timely decision making, which is  preferable in a dynamic environment with limited computing resources. Moreover, even for some stationary scenarios as shown in the numerical results, the ``Champion Solution'' can still achieve a performance comparable to the one using optimality in expectation.

Future work is aiming at generalizing the sufficient existence condition and extending the idea of champion solution to a wider class of stochastic optimization problems.

\bibliographystyle{plain}
\bibliography{JF_DST}

\end{document}